\documentclass[11pt]{amsart}
\usepackage[utf8]{inputenc}
\usepackage{amsmath}
\usepackage{amssymb}
\usepackage{amsfonts}
\usepackage{nicefrac}
\usepackage{amsthm}
\usepackage{mathrsfs}
\usepackage{mathtools}
\usepackage{subfig}
\usepackage{tikz}
\usetikzlibrary{patterns, snakes}
\usepackage{float}
\usepackage{enumerate}
\usepackage{graphicx}
\usepackage{comment}
\newcommand{\N}{\mathbb{N}}

\newcommand{\p}{\mathcal{P}}
\newcommand{\LG}{\operatorname{lg}}
\newcommand{\s}{\mathcal{S}}
\newcommand{\f}{\mathcal{F}}
\newcommand{\Ll}{\mathscr{L}}
\newcommand{\SCP}{sequentially congruent partition}
\newcommand{\tail}{\operatorname{Tail}}
\newcommand{\lcm}{\operatorname{lcm}}

\theoremstyle{plain}
\makeatletter
\newtheorem*{rep@theorem}{\rep@title}
\newcommand{\newreptheorem}[2]{%
\newenvironment{rep#1}[1]{%
 \def\rep@title{#2 \ref{##1}}%
 \begin{rep@theorem}}%
 {\end{rep@theorem}}}
\makeatother

\newtheorem{theorem}{Theorem}[section]
\newtheorem{corollary}[theorem]{Corollary}
\newtheorem{proposition}[theorem]{Proposition}
\newtheorem*{proposition*}{Proposition}
\newtheorem{lemma}[theorem]{Lemma}
\newtheorem*{lemma*}{Lemma}
\newreptheorem{theorem}{Theorem}
\newreptheorem{corollary}{Corollary}
\newreptheorem{proposition}{Proposition}
\newreptheorem{lemma}{Lemma}

\theoremstyle{definition}
\newtheorem{definition}[theorem]{Definition}
\newtheorem{remark}[theorem]{Remark}
\newtheorem{example}[theorem]{Example}

\usepackage{graphicx}
\usepackage{young}
\usepackage{ytableau}

\begin{document}

\title{Bijections, generalizations, and other properties of sequentially congruent partitions}

\author{Ezekiel Cochran}
\address{LeTourneau University\newline
Longview, TX, 75602, U.S.A.}
\email{ezekielcochran@letu.edu}

\author{Madeline Locus Dawsey}
\address{Department of Mathematics\newline
University of Texas at Tyler\newline
Tyler, TX 75799, U.S.A.}
\email{mdawsey@uttyler.edu}

\author{Emma Harrell}
\address{Mount Holyoke College\newline
South Hadley, MA, 01075,  U.S.A.}
\email{harre22e@mtholyoke.edu}

\author{Samuel Saunders}
\address{University of Texas at Tyler\newline
Tyler, TX 75799, U.S.A.}
\email{ssaunders@uttyler.edu}



\date{}

\begin{abstract}
Recently, Schneider and Schneider defined a new class of partitions called sequentially congruent partitions, in which each part is congruent to the next part modulo its index, and they proved two partition bijections involving these partitions.  We introduce a new partition notation specific to sequentially congruent partitions which allows us to more easily study these bijections and their compositions, and we reinterpret them in terms of Young diagram transformations.  We also define a generalization of sequentially congruent partitions, and we provide several new partition bijections for these generalized sequentially congruent partitions.  Finally, we investigate a question of Schneider--Schneider regarding how sequentially congruent partitions fit into Andrews' theory of partition ideals.  We prove that the maximal partition ideal of sequentially congruent partitions has infinite order and is therefore not linked, and we identify its order 1 subideals.
\end{abstract}

\maketitle

\section{Introduction}\label{section_intro}

A partition $\lambda$ of a nonnegative integer $n$ is a sequence $(\lambda_1,\lambda_2,\dots,\lambda_r)$ of positive integers, called the parts of $\lambda$, such that $\sum_{i=1}^r\lambda_i=n$ and $\lambda_1\geq\lambda_2\geq\dots\geq\lambda_r$. Sometimes it is useful to refer to a partition $\lambda$ by its frequency notation $\lambda=\langle 1^{f_1}, 2^{f_2}, 3^{f_3}, \dots \rangle$, where the frequency $f_i$ (also called the multiplicity) is the number of times the part $i$ appears in $\lambda$. Throughout this paper, frequency notation does not include parts with frequency zero, unless we explicitly say otherwise.  The size of a partition $\lambda$ of $n$ is $|\lambda|=n$, and the length of $\lambda$, denoted $\ell(\lambda)$, is the number of parts that comprise $\lambda$.

The Young diagram of a partition $\lambda$ is an array of $|\lambda|$ nodes (or boxes) with $\ell(\lambda)$ rows, where the $i$th row has $\lambda_i$ nodes. The conjugate of a partition is obtained by reflecting the Young diagram over its main (upper left to lower right) diagonal. For example, the partition $(7,5,5,4,1)$ has the Young diagram below, and its conjugate is shown to the right.

\begin{figure}[H]
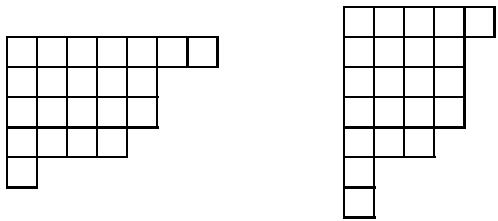
%
\centering
\subfloat{{\ytableausetup
{mathmode, boxsize=1em, centertableaux}
\begin{ytableau}
~ & & & & & &\\
  & & & &\\
  & & & &\\
  & & &\\
  \\
\end{ytableau}}}%
\qquad\qquad
\subfloat{{\ytableausetup
{mathmode, boxsize=1em, centertableaux}
\begin{ytableau}
~ & & & &\\
  & & &\\
  & & &\\
  & & &\\
  & &\\
  \\
  \\
\end{ytableau}}}%
\caption{Young diagrams of $\lambda=(7,5,5,4,1)$ and its conjugate}
\label{conjugate_figure}
\end{figure}

Counting the number of partitions of a certain size and type is a question that has intrigued mathematicians for centuries. A partition bijection is a bijection between two sets of partitions, which implies that the two sets contain the same number of partitions. The first famous partition bijection was proved by Euler in the 1700s: the number of partitions of a natural number $n$ into odd parts is equal to the number of partitions of $n$ into distinct parts. In the late 19th and early 20th century, Rogers and Ramanujan discovered the following two q-series identities, now known as the Rogers--Ramanujan identities, which lead to two more sophisticated partition bijections:
\begin{align*}
    \sum_{n=0}^\infty \frac{q^{n^2}}{(1-q)(1-q^2)\cdots(1-q^n)}=\prod_{n=1}^\infty \frac{1}{(1-q^{5n-1})(1-q^{5n-4})},\\
    \sum_{n=0}^\infty\frac{q^{n^2+n}}{(1-q)(1-q^2)\cdots(1-q^n)}=\prod_{n=1}^\infty \frac{1}{(1-q^{5n-2})(1-q^{5n-3})}.
\end{align*}
The first identity implies that the number of partitions of a natural number $n$ in which adjacent parts differ by at least 2 equals the number of partitions of $n$ into parts congruent to 1 or 4 modulo 5.  The second identity implies that the number of partitions of $n$ in which adjacent parts differ by at least 2 and whose smallest part is at least 2 is the same as the number of partitions of $n$ into parts congruent to 2 or 3 modulo 5.  Partitions in which adjacent parts differ by at least 2, as mentioned in the bijection implied by the first identity, are sometimes called Rogers--Ramanujan partitions. By rearranging the Young diagrams, it is straightforward to see that Rogers--Ramanujan partitions are also in bijection with partitions with no parts below the Durfee square, which is the largest square that fits in the top left corner of a Young diagram. For example, the partition $\lambda=(7,6,3,3,1)$ shown below has a $3\times3$ Durfee square. 

\begin{figure}[H]
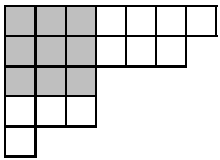

\centering
\ytableausetup
{mathmode, boxsize=1em, centertableaux}
\begin{ytableau}
*(lightgray) & *(lightgray) & *(lightgray) & 
&  & & \\
*(lightgray) & *(lightgray) & *(lightgray) & 
&  &  \\
*(lightgray) & *(lightgray) & *(lightgray) \\
& &\\
\\
\end{ytableau}
\caption{Durfee square of $\lambda=(7,6,3,3,1)$}
\end{figure}
\noindent A partition with no parts below its Durfee square can also be thought of as a partition whose smallest part is greater than or equal to the length of the partition.

Recently, Schneider and Schneider \cite{schneider1} defined a new type of partition.
\begin{definition}
A \emph{sequentially congruent partition} $\lambda = (\lambda_1, \dots, \lambda_r)$ is a partition such that
\begin{enumerate}
    \item $\lambda_i \equiv  \lambda_{i + 1} \pmod i$ for all $1 \leq i < r$, and
    \item $\lambda_r \equiv 0 \pmod r$.
\end{enumerate}
\end{definition}
\noindent We denote the set of sequentially congruent partitions by $\s$.  An example of a sequentially congruent partition is $(21,16,14,8)$, since $1\vert (21-16)$,  $2\vert (16-14)$, $3\vert (14-8)$, and $4\vert 8$. Note that the conjugate of a sequentially congruent partition is of the form $\left<1^{f_1},2^{f_2}, 3^{f_3},\dots\right>$, where $f_i\equiv 0 \pmod i$. That is to say, the part $i$ occurs in the partition a multiple of $i$ times. Schneider and Schneider refer to the conjugates of sequentially congruent partitions as \emph{frequency congruent partitions} \cite{schneider1}.

In the two papers so far on sequentially congruent partitions \cite{schneider1,schneider2}, several partition bijections between the set of sequentially congruent partitions and other sets of partitions have been found, many of which use the set of frequency congruent partitions as an intermediary. Some open questions arise in these papers, such as a description of the composition of these bijections, how to define analogous bijections with generalized sequentially congruent partitions, and the connection, if it exists, between sequentially congruent partitions and partition ideals as defined by Andrews \cite{andrews1974} (for a more concise treatment of partition ideals, see \cite[Chapter 8]{andrews}). We answer these questions in this work.

First, we introduce some notation that will simplify our study of sequentially congruent partitions.  For a partition $\pi$, we consider $\pi_i = 0$ for all $i > \ell (\pi)$. We define an operation $\star$ such that for any two partitions $\lambda$ and $\gamma$, we have that $\lambda \star \gamma := \mu$, where $\mu = (\mu_1, \mu_2, \dots)$ is the partition such that $\mu_i = \lambda_i + \gamma_i$ for all $i \in \N$. In other words, this operation ``adds" two partitions, componentwise from the left. For example, $(5, 3, 2, 2) \star (3, 2, 1) = (8, 5, 3, 2)$. Note that $ |\lambda \star \gamma| =|\lambda| + |\gamma|$ and $\ell(\lambda\star\gamma)=\max\{\ell(\lambda),\ell(\gamma)\}$. We also define $c(\lambda_1,\dots,\lambda_r)=(c\lambda_1,\dots,c\lambda_r)$, so that $c\lambda=\underbrace{\lambda\star\lambda\star\dots\star\lambda}_{\text{$c$ times}}$.
Using $\star$, we observe a new, simple way to write a sequentially congruent partition.

\begin{theorem}
\label{c_notation_theorem}
A partition $\lambda$ is sequentially congruent if and only if it can be written uniquely in the form
\begin{align*}
    \lambda &= c_1(1) \star c_2(2,2) \star c_3(3,3,3) \star \cdots
\end{align*}
for some nonnegative integers $c_1, c_2, c_3, \dots$, where finitely many $c_i$ are nonzero.
\end{theorem}

\begin{remark}
A sequentially congruent partition $\lambda=c_1(1)\star c_2(2,2)\star c_3(3,3,3)\star\cdots$ can be written in standard notation as $\lambda=(c_1+2c_2+3c_3+\dots+r c_r, 2c_2+3c_3+\dots+r c_r,\dots,r c_r)$, where $r$ is the largest integer so that $c_r$ is nonzero.
\end{remark}

Theorem \ref{c_notation_theorem} leads to a new partition notation that is surprisingly natural for sequentially congruent partitions and simplifies the results and proofs in this paper. This notation, which we call \emph{$c$-notation}, will be discussed in more detail in Section \ref{section_c-notation}. 

Schneider, Sellers, and Wagner \cite{schneider2} prove a bijection between sequentially congruent partitions and partitions whose parts are squares. This bijection also follows easily from Theorem \ref{c_notation_theorem}.

\begin{corollary}
\label{bijection_into_squares_corollary}
The number of sequentially congruent partitions of size $n$ equals the number of partitions of $n$ whose parts are perfect squares.
\end{corollary}
For a set $Q$ of partitions, let $p(Q,n)$ be the number of partitions in $Q$ of size $n$.  It is known that partitions into squares, and thus sequentially congruent partitions, have the generating function
\begin{equation}\label{generating_function_squares}
\sum_{n = 0}^{\infty} p(\s, n)q^n = \prod_{n=1}^\infty \frac{1}{1-q^{n^2}}.
\end{equation}

Sums of squares have been studied for centuries, and although \eqref{generating_function_squares} does not have nice modular transformation properties like many other generating functions related to partitions, asymptotics are known.  For example, see \cite{B,G,HR,H,HS,L} for background and more details on partitions into squares.

In \cite{schneider1}, Schneider and Schneider define the bijections $\pi$ and $\sigma$ between $\p$, the set of all partitions, and $\s$. The bijection $\pi:\p \to\s$ maps partitions of size $n$ to sequentially congruent partitions with largest part $n$. For a partition $\lambda=(\lambda_1,\lambda_2,\dots,\lambda_r)\in \p$ of size $n$, they define $\pi(\lambda)=\lambda'$, where $\lambda' = (\lambda_1',\lambda_2',\dots,\lambda_r')$ is the partition with
\begin{equation}\label{pi}
\lambda_i'=i\lambda_i + \sum_{j=i+1}^r \lambda_j
\end{equation}
for all $1\leq i\leq r$.  The bijection $\sigma:\s \to \p$ maps sequentially congruent partitions with largest part $n$ to partitions of size $n$. For a sequentially congruent partition $\phi=(\phi_1,\phi_2,\dots,\phi_r)$, they define
\begin{equation}\label{sigma}
\sigma(\phi)= \big< 1^{\phi_1-\phi_2},2^{(\phi_2-\phi_3)/2},3^{(\phi_3-\phi_4)/3},\dots, r^{\phi_r/r} \big>.
\end{equation}
We now give examples of the maps $\pi$ and $\sigma$. 
For the partition $(12,8,4,3,3)\in \p$ of size $30$, we have
\begin{align*}
    &\pi((12,8,4,3,3))=(30,26,18,15,15)\in \s\\
\shortintertext{and}\\ 
&\sigma((30,26,18,15,15))=\langle 1^4, 2^4, 3^1, 5^3 \rangle\in \p.
\end{align*}
Schneider and Schneider find that the composition $\sigma\circ\pi:\p\rightarrow\p$ is equivalent to conjugation \cite{schneider1}.  Based on their result, the following proposition becomes apparent.  While conjugation is traditionally defined in terms of Young diagrams reflected about the main diagonal, the composition $\sigma \circ \pi$ gives an alternate definition of conjugation in terms of the parts of partitions.

\begin{proposition}
\label{conjugate_closed_form_proposition}
The conjugate of any partition $\lambda = (\lambda_1,\lambda_2,\dots,\lambda_r)$ is the partition \[\big\langle 1^{\lambda_1-\lambda_2},2^{\lambda_2-\lambda_3},\dots ,(r-1)^{\lambda_{r-1}-\lambda_r},r^{\lambda_r}\big\rangle.\]
\end{proposition}

\begin{corollary}
\label{self_conjugate_corollary}
A partition $\lambda = (\lambda_1,\lambda_2,\dots,\lambda_r)$ is self-conjugate if and only if \[\lambda_i= \begin{cases}
r & \; \; \text{ when }\;i\leq \lambda_r,\\
r-1 & \; \; \text{ when }\; \lambda_r < i \leq \lambda_{r-1},\\
\vdots&\\
2 & \; \; \text{ when }\; \lambda_3<i\leq \lambda_2,\\
1 & \; \; \text{ when }\; \lambda_2<i\leq \lambda_1.
\end{cases}\]
\end{corollary}
\noindent The proofs of Proposition \ref{conjugate_closed_form_proposition} and Corollary \ref{self_conjugate_corollary} are immediate from the definitions of $\pi$ and $\sigma$.

Schneider and Schneider also briefly investigate the composition $\pi \circ \sigma$ \cite{schneider1} and observe that $\pi\circ \sigma$ is not equivalent to conjugation. Instead, they suggest that this composition is some analogue of conjugation specific to sequentially congruent partitions.  We identify this analogue in the following theorem.

\begin{theorem}
\label{pi_composed_sigma_theorem}
Let $\phi=(\phi_1, \phi_2, \dots, \phi_r)$ be a sequentially congruent partition. Then we have that $$(\pi\circ\sigma)(\phi)=\left\langle \left( \sum_{j=1}^{1} \sum_{i = j}^{r} \frac{\phi_i - \phi_{i + 1}}{i} \right)^{\phi_1 - \phi_2}, \left( \sum_{j=1}^{2} \sum_{i = j}^{r} \frac{\phi_i - \phi_{i + 1}}{i} \right)^{\frac{\phi_2 - \phi_3}{2}}, \dots, \left( \sum_{j=1}^{r} \sum_{i = j}^{r} \frac{\phi_i - \phi_{i + 1}}{i} \right)^{\frac{\phi_r}{r}} \right\rangle.$$
\end{theorem}

Through the lens of Theorem \ref{c_notation_theorem} and $c$-notation, we will more clearly describe $\pi\circ\sigma$ in Section \ref{section_bijections}.

In Section \ref{section_c-notation}, we prove Theorem \ref{c_notation_theorem} and Corollary \ref{bijection_into_squares_corollary}.  In Section \ref{section_bijections}, we discuss bijections between sequentially congruent partitions and other classes of partitions, and we prove Theorem \ref{pi_composed_sigma_theorem}.  In Section \ref{section_generalizations}, we introduce generalizations of sequentially congruent partitions and of the bijections discussed in Section \ref{section_bijections}.  In Section \ref{section_ideals}, we explore the question of Schneider--Schneider regarding how sequentially congruent partitions fit into Andrews' theory of partition ideals.

\section{$c$-Notation}\label{section_c-notation}

In this section, we prove Theorem \ref{c_notation_theorem} and then introduce notation that is useful for working with sequentially congruent partitions. This new notation will be used frequently throughout the paper.

\begin{proof}[Proof of Theorem \ref{c_notation_theorem}]
Let $\lambda$ be a partition of the form \begin{align*}
    \lambda
    &=c_1(1) \star c_2(2,2) \star c_3(3,3,3) \star \dots \star c_r(\underbrace{r,\dots,r}_{r\text{ times}}) \\ 
    &= (c_1+2c_2+\dots+rc_r, 2c_2+\dots+rc_r,\dots,rc_r).
\end{align*}
It is clear that $\ell(\lambda)=r$ and that $rc_r\equiv 0 \pmod r$. For any $\lambda_k$ with $1\leq k < r$, we have that $\lambda_k=kc_k+(k+1)c_{k+1}+ \dots +rc_r$ and $\lambda_{k+1}=(k+1)c_{k+1}+ \dots +rc_r$, and hence $\lambda_k-\lambda_{k+1}=kc_k$. Then $\lambda_k\equiv \lambda_{k+1}\pmod k$. Therefore, $\lambda$ is a sequentially congruent partition.

Now let $\lambda = (\lambda_1,\lambda_2,\dots,\lambda_r)$ be a sequentially congruent partition. Then $\lambda_r\equiv 0\pmod r$ and $\lambda_k\equiv \lambda_{k+1} \pmod k$ for all $1\leq k < r$. 
Thus $\lambda_r=r c_r$, $\lambda_{r-1}-\lambda_r=(r-1)c_{r-1}$, \dots, $\lambda_2-\lambda_3=2c_2$, and $\lambda_1-\lambda_2 =c_1$ for some nonnegative integers $c_1,\dots,c_r$. 
Then note that we can rewrite the parts as follows:
\begin{align*}
    \lambda_{r-1}&=(r-1)c_{r-1}+\lambda_r \\
    &= (r-1)c_{r-1}+ rc_r\\
    \lambda_{r-2}&=(r-2)c_{r-2}+\lambda_{r-1} \\
    &= (r-2)c_{r-2}+ (r-1)c_{r-1}+ rc_r\\
& \; \; \; \; \; \; \; \;\; \; \;\; \; \; \;\; \; \; \; \; \vdots \\
    \lambda_{2}&=2c_2+3c_3+\dots+(r-1)c_{r-1}+ rc_r\\
    \lambda_{1}&=c_1+2c_2+3c_3+\dots+(r-1)c_{r-1}+ rc_r.
\end{align*}
Therefore 
\begin{align*}
    \lambda &= (c_1+2c_2+\dots+rc_r, 2c_2+\dots+rc_r,\dots,rc_r)\\
    &=c_1(1) \star c_2(2,2) \star c_3(3,3,3) \star \dots \star c_r(\underbrace{r,\dots,r}_{r\text{ times}}).
\end{align*}
Thus any sequentially congruent partition can be written in the desired form.

Finally, suppose
\begin{align*}
\lambda&=c_1(1) \star c_2(2,2) \star c_3(3,3,3) \star \dots \star c_r(\underbrace{r,\dots,r}_{r\text{ times}})\\
&=c_1'(1) \star c_2'(2,2) \star c_3'(3,3,3) \star \dots \star c_r'(\underbrace{r,\dots,r}_{r\text{ times}})
\end{align*}
are two representations of the same partition $\lambda$ in $c$-notation. Observe that $\lambda_r=rc_r$ and $\lambda_r=rc_r'$. Hence $c_r=c_r'$. Now, let $n\geq 1$ be an integer, and assume that $c_{r-k}=c'_{r-k}$ for all $0\leq k<n$. Then \begin{align*}
\lambda_{r-n}&=(r-n)c_{r-n}+(r-n+1)c_{r-n+1}+\dots+rc_r\\
&=(r-n)c_{r-n}+(r-n+1)c_{r-n+1}'+\dots+rc_r'\\
\intertext{by the induction hypothesis, and}\lambda_{r-n}&=(r-n)c_{r-n}'+(r-n+1)c_{r-n+1}'+\dots+rc_r'
\end{align*}
by assumption, so $c_{r-n}=c_{r-n}'$.  Thus, by induction, we have that $c_{r-n}=c_{r-n}'$ for all $0\leq n< r$, and the representation is unique.
\end{proof}

\begin{definition}
Since the representation in Theorem \ref{c_notation_theorem} is unique, we can denote any sequentially congruent partition $\lambda=c_1(1) \star c_2(2,2) \star c_3(3,3,3) \star \dots \star c_r (\underbrace{r, \dots, r}_{r \text{ times}})$ simply by $\lambda=[c_1,c_2,\dots,c_r]$. We refer to this notation as \textit{$c$-notation}, and we distinguish it from standard notation (which uses parentheses) and frequency notation (which uses angle brackets) by using square brackets.
\end{definition}
For example, we write the partition $(8, 6, 4, 4)$ in $c$-notation as $$(8,6,4,4)=2(1) \star 1(2,2) \star 0(3,3,3) \star 1(4,4,4,4) = [2, 1, 0, 1].$$
Note that a partition $\lambda = [c_1, \dots, c_r]$ has length $r$ and size $c_1 + 4c_2 + \cdots + r^2c_r = \sum_{i = 1}^{r}i^2 c_i$. Furthermore, $\lambda_i=\sum_{j=i}^{r} j c_j$ for each $1\leq i\leq r$. Then observe that $\lambda_i-\lambda_{i+1}=ic_i$, and so $c_i=(\lambda_i-\lambda_{i+1})/i$ for each $i$. Since we consider $\lambda_i=0$ for all $i>\ell(\lambda)$, we have that $c_r=\lambda_r/r$ and $c_i =0$ for all $i > r$. 

The Young diagram of a sequentially congruent partition is composed of squares of varying sizes, as observed by Schneider--Sellers--Wagner \cite{schneider2}. This can be seen in the example below, which is the Young diagram for the sequentially congruent partition $(16, 15, 11, 5, 5)$.

\ytableausetup{mathmode, boxsize=1em, centertableaux}
\begin{figure}[H]
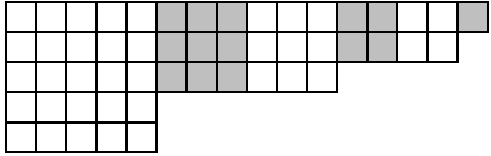

\centering
\begin{ytableau}
~ &  & & & &*(lightgray) &*(lightgray) &*(lightgray) & & & &*(lightgray) &*(lightgray) & & &*(lightgray) \\
~  & & & & &*(lightgray) &*(lightgray) &*(lightgray) & & & &*(lightgray) &*(lightgray) & &\\
~  & & & & &*(lightgray) &*(lightgray) &*(lightgray)& & & \\
~  & & & & \\
& & & &  \\
\end{ytableau}
\caption{Young diagram of $(16, 15, 11, 5, 5) \in \s$}
\end{figure}
In fact, the number of $i\times i$ squares corresponds to the value of $c_i$ in the $c$-notation representation of the partition. For the above example, note that $(16, 15, 11, 5, 5)=[1,2,2,0,1]$, and for each $1\leq i \leq 5$, the $i\times i$ square occurs $c_i$ times.  In general, the Young diagram of a sequentially congruent partition looks like this:

\begin{figure}[H]
\begin{center}
	\begin{tikzpicture} [scale = 0.3]
	
		\node (bigdots) at (-13, 0) {$\cdots$};
		
    		\draw (-12,0) grid (-7,-5);
    		\node (dots) at (-6, -2.5) {$\cdots$};
		\draw (-5, 0) grid (0, -5);
    
    		\draw [
    		thick,
    		decoration={
		        brace,
		        mirror,
		        raise = 0.1cm,
		    },
		    decorate
		    ] (-12, -5) -- (0, -5)
		    node [midway, anchor = north, yshift = -1mm, xshift = -1mm] {$c_5$ times};
	
    		\draw (0,0) grid (4,-4);
    		\node (dots) at (5, -2) {$\cdots$};
		\draw (6, 0) grid (10, -4);
    
    		\draw [
    		thick,
    		decoration={
		        brace,
		        mirror,
		        raise = 0.1cm,
		    },
		    decorate
		    ] (0.2, -4) -- (10, -4)
		    node [midway, anchor = north, yshift = -1mm, xshift = -1mm] {$c_4$ times};
		    
    		\draw (10,0) grid (13,-3);
    		\node (dots) at (14, -1.5) {$\cdots$};
		\draw (15, 0) grid (18, -3);
    
    		\draw [
    		thick,
    		decoration={
		        brace,
		        mirror,
		        raise = 0.1cm,
		    },
		    decorate
		    ] (10.2, -3) -- (18, -3)
		    node [midway, anchor = north, yshift = -1mm, xshift = -1mm] {$c_3$ times};
		    
    		\draw (18,0) grid (20,-2);
    		\node (dots) at (21, -1) {$\cdots$};
		\draw (22, 0) grid (24, -2);
    
    		\draw [
    		thick,
    		decoration={
		        brace,
		        mirror,
		        raise = 0.1cm,
		    },
		    decorate
		    ] (18.2, -2) -- (24, -2)
		    node [midway, anchor = north, yshift = -1mm, xshift = -1mm] {$c_2$ times};
		    
    		\draw (24,0) grid (25,-1);
    		\node (dots) at (26, -0.5) {$\cdots$};
		\draw (27, 0) grid (28, -1);
    
    		\draw [
    		thick,
    		decoration={
		        brace,
		        mirror,
		        raise = 0.1cm,
		    },
		    decorate
		    ] (24.2, -1) -- (28, -1)
		    node [midway, anchor = north, yshift = -1mm, xshift = 1mm] {$c_1$ times};
		
	\end{tikzpicture}
\end{center}
\caption{Young diagram of a partition $[c_1,c_2,\dots]\in\s$}
\end{figure}
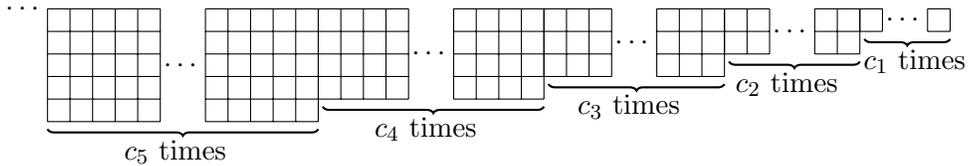

We now prove the bijection given in \cite{schneider2} between sequentially congruent partitions and partitions whose parts are squares using $c$-notation.

\begin{proof}[Proof of Corollary \ref{bijection_into_squares_corollary}]
Let $\s(n)$ denote the set of \SCP s of size $n$, and let $\p_{\N^2}(n)$ denote the set of partitions of $n$ whose parts are perfect squares. We have from Theorem \ref{c_notation_theorem} that any partition $\phi\in\s$ can be uniquely written in $c$-notation as $\phi=[c_1, \dots, c_r]$. Additionally, any partition $\lambda \in \p_{\N^2}$ can be uniquely written as $\lambda = \big\langle {(1^2)}^{d_1}, {( 2^2 )}^{d_2}, \dots, {( r^2 )}^{d_r}\big\rangle$, for some nonnegative integers $d_1,\dots,d_r$. Because both of these descriptions are unique, and each $c_i$ can take precisely the same values as $d_i$, we have a bijection $\psi: \s \to \p_{\N^2}$ defined by $$\psi \left( [c_1, c_2, \dots, c_r] \right) := \left\langle {(1^2)}^{c_1}, {( 2^2 )}^{c_2}, \dots, {( r^2 )}^{c_r} \right\rangle.$$
Recall that the size of the partition $\phi=[c_1, \dots, c_r]\in\s$ is $\sum_{i=1}^r i^2c_i$. Furthermore, the size of the output of $\psi$ is $(1^2)c_1 +(2^2)c_2 + \dots + ( r^2 )c_r = \sum_{i=1}^r i^2c_i$. Therefore, $\psi$ is a size-preserving bijection.
\end{proof}

This bijection is easily understood through Young diagrams. We have shown that sequentially congruent partitions have Young diagrams composed of squares. To map a sequentially congruent partition to a partition whose parts are squares, we simply transform each of the $i\times i$ squares into a row with $i^2$ boxes, as shown below.

\begin{center}
\begin{figure}[H]%
\subfloat{{\ytableausetup
{mathmode, boxsize=.99em, centertableaux}
\begin{ytableau}
  ~ & & & *(lightgray) & *(lightgray) & *(lightgray) & & & *(lightgray) & & *(lightgray) \\
  ~ & & & *(lightgray) & *(lightgray) & *(lightgray) & & \\
  ~ & & & *(lightgray) & *(lightgray) & *(lightgray) \\
\end{ytableau}}}%
$\hspace{.25cm}\xmapsto{\;\;\psi\;\;}$
\subfloat{{\ytableausetup
{mathmode, boxsize=1em, centertableaux}
\begin{ytableau} 
~ & & & & & & & & \\
*(lightgray) &*(lightgray) &*(lightgray) & *(lightgray)& *(lightgray)& *(lightgray)& *(lightgray)&*(lightgray) & *(lightgray)\\
~ & & &\\
*(lightgray)\\ 
\\
*(lightgray)
\end{ytableau}}}%
\label{bijection_into_squares_figure}
\caption{The bijection $\psi: \s \to \p_{\N^2}$}
\end{figure}
\end{center}

\section{Sequentially Congruent Partition Bijections}\label{section_bijections}

Let $\p(n)$ denote the set of partitions of $n$, and let $\s_{\LG = n}$ denote the set of sequentially congruent partitions with largest part $n$. Recall that Schneider and Schneider \cite{schneider1} defined the bijections $\pi: \p(n) \to \s_{\LG = n}$ and $\sigma: \s_{\LG = n} \to \p(n)$ as in \eqref{pi} and \eqref{sigma}.  Here we rewrite $\pi$ and $\sigma$ using $c$-notation and describe how they transform the Young diagrams of partitions. 

\begin{proposition}
\label{pi_c_representation_proposition}
Let $\lambda=(\lambda_1, \lambda_2, \dots, \lambda_r)$ be any partition. We have that $\pi(\lambda) = [c_1, \dots, c_r]$, where $c_i = \lambda_i - \lambda_{i + 1}$ for all $1 \leq i < r$ and $c_r = \lambda_r$.
\end{proposition}

\begin{proof}
Let $\lambda=(\lambda_1,\lambda_2,\dots,\lambda_r)$ be a partition. Then we have that $\pi(\lambda)=\lambda'$, where
\begin{align*}
    \lambda'_i &= i \lambda_i + \sum_{j = i + 1}^{r} \lambda_j\\
    &= i \lambda_i + \lambda_{i + 1} + \dots + \lambda_{r - 1} + \lambda_r.
    \end{align*}
    We rewrite each coefficient of 1 and simplify as follows.
    \begin{align*}
    \lambda'_i&= i \lambda_i + (-i + (i + 1)) \lambda_{i + 1} + \dots + (-(r - 2) + (r - 1)) \lambda_{r - 1} + (-(r - 1) + r)\lambda_r \\
    &= i\lambda_i - i\lambda_{i + 1} + (i + 1)\lambda_{i + 1} - (i + 1)\lambda_{i + 2} + \dots + (r - 1)\lambda_{r - 1} - (r - 1)\lambda_r + r\lambda_r \\
    &= i(\lambda_i - \lambda_{i+1}) + (i + 1)(\lambda_{i + 1} - \lambda_{i + 2}) + \dots + (r - 1)(\lambda_{r - 1} - \lambda_r) + r \lambda_r \\
    &= \sum_{j = i}^{r-1} j(\lambda_j - \lambda_{j + 1}) + r\lambda_r.
    \end{align*}
    Then we define $c_i = \lambda_i - \lambda_{i + 1}$ for all $1 \leq i <r$ and $c_r=\lambda_r$, so that
    \begin{align*}
    \lambda'_i &= \sum_{j = i}^{r-1} jc_j + rc_r = \sum_{j = i}^{r} jc_j
\end{align*}
for all $1\leq i<r$ and $\lambda_r'=rc_r$.  Therefore, by Theorem \ref{c_notation_theorem}, we have that $\lambda' = [c_1, \dots, c_r]$.
\end{proof}

To see why $\pi:\p(n)\to\s_{\LG=n}$ maps size to largest part, we can think of its action on the Young diagram as transforming each column into the top row of a square in the resulting partition.
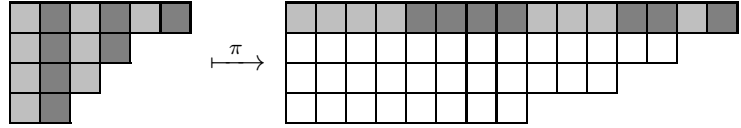
\begin{center}
\begin{figure}[H]%
\subfloat{{
\begin{ytableau}
*(lightgray) & *(gray) & *(lightgray) & *(gray) & *(lightgray) & *(gray) \\
*(lightgray) & *(gray) & *(lightgray) & *(gray) \\
*(lightgray) & *(gray) & *(lightgray)\\
*(lightgray) & *(gray)
\end{ytableau}}}%
$\hspace{.25cm}\xmapsto{\;\;\pi\;\;}$
\subfloat{{
\begin{ytableau} 
*(lightgray)&*(lightgray) &*(lightgray) &*(lightgray) &*(gray)&*(gray)&*(gray)&*(gray)&*(lightgray) &*(lightgray) &*(lightgray) &*(gray) & *(gray)&*(lightgray) &*(gray)\\
~& & & & & & & & & & & &  \\
~& & & & & & & & & &\\
~& & & & & & &
\end{ytableau}}}%
\label{pi_size_figure}
\end{figure}
\end{center}
Alternatively, we can describe this action by ``stretch": the map $\pi$ transforms each $i\times 1$ column in the Young diagram into an $i\times i$ square.
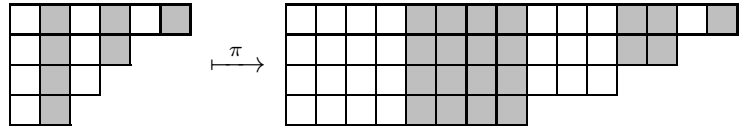
\begin{center}
\begin{figure}[H]%
\subfloat{{
\begin{ytableau}
~ & *(lightgray) & & *(lightgray) & & *(lightgray) \\
~ & *(lightgray) & & *(lightgray) \\
~ & *(lightgray) &\\
~ & *(lightgray)
\end{ytableau}}}%
$\hspace{.25cm}\xmapsto{\;\;\pi\;\;}$
\subfloat{{
\begin{ytableau} 
~& & & &*(lightgray)&*(lightgray)&*(lightgray)&*(lightgray)& & &&*(lightgray) & *(lightgray)& &*(lightgray)\\
~& & & &*(lightgray)&*(lightgray)&*(lightgray)&*(lightgray)& & &&*(lightgray) & *(lightgray)\\
~& & & &*(lightgray)&*(lightgray)&*(lightgray)&*(lightgray)& & &\\
~& & & &*(lightgray)&*(lightgray)&*(lightgray)&*(lightgray)
\end{ytableau}}}%
\label{pi_conjugacy_figure}
\end{figure}
\end{center}

We now turn to the map $\sigma: \s_{\LG = n} \to \p(n)$. Since any partition $\phi = (\phi_1, \phi_2, \dots, \phi_r)\in\s$ can be written in $c$-notation as $\phi = [c_1, c_2, \dots, c_r]$ with $c_i = (\phi_i - \phi_{i + 1})/i$ for all $1 \leq i <r$ and $c_r=\phi_r/r$, we immediately obtain the following description of $\sigma$ using $c$-notation. 
\begin{proposition}
\label{c_sigma_prop}
Let $\phi=[c_1,c_2,\dots,c_r]$ be any sequentially congruent partition.  We have that $\sigma(\phi) = \left<1^{c_1}, 2^{c_2}, \dots, r^{c_r} \right>$.
\end{proposition}

In the Young diagram, we can think of $\sigma$ as treating the top row of each square as a part in the resulting partition.
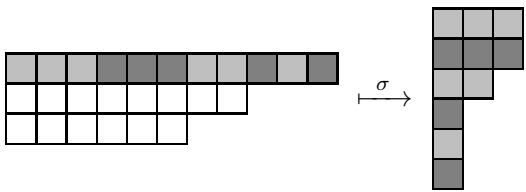
\begin{figure}[H]
\centering
\subfloat{{
\begin{ytableau}
*(lightgray) & *(lightgray) & *(lightgray) & *(gray) & *(gray) & *(gray) & *(lightgray) & *(lightgray) & *(gray) & *(lightgray) & *(gray)\\
~ & & & & & & & \\
~ & & & & & \\
\end{ytableau}}}%
$\hspace{.25cm}\xmapsto{\;\; \sigma \;\;}$
\subfloat{{
\begin{ytableau} 
*(lightgray) & *(lightgray) & *(lightgray) \\
*(gray)&*(gray)&*(gray)\\
 *(lightgray) & *(lightgray) \\
*(gray)\\
*(lightgray)\\
*(gray)\\
\end{ytableau}}}%
\end{figure}
Alternatively, we can describe this action by ``squish-flip": each $i\times i$ square is reduced to an $i\times 1$ column (``squish"), and then we take the conjugate (``flip").
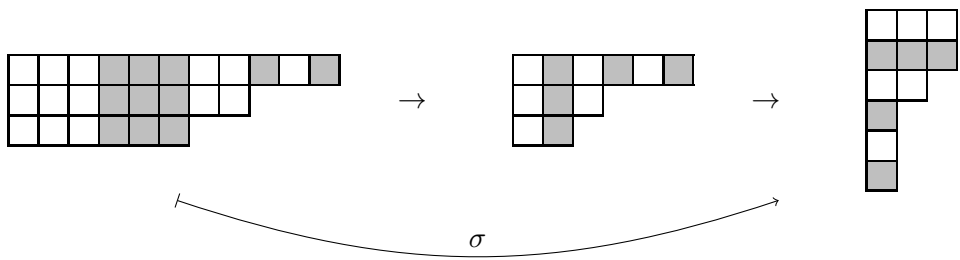
\begin{figure}[H]
\centering
\subfloat{{
\begin{ytableau}
~ & & &*(lightgray) &*(lightgray) &*(lightgray) & & &*(lightgray) & & *(lightgray)\\
~ & & &*(lightgray) &*(lightgray) &*(lightgray) & & \\
~ & & &*(lightgray) &*(lightgray) &*(lightgray) \\
\end{ytableau}}}%
\qquad $\rightarrow{}$ \qquad
\subfloat{{
\begin{ytableau} 
~&*(lightgray)& &*(lightgray) & &*(lightgray)\\
~&*(lightgray)& \\
~&*(lightgray)\\
\end{ytableau}}}%
\qquad $\rightarrow{}$ \qquad
\subfloat{{
\begin{ytableau} 
~ & & \\
*(lightgray)&*(lightgray)&*(lightgray)\\
& \\
*(lightgray)\\
\\
*(lightgray)\\
\end{ytableau}}}
\label{squish_flip_figure}
\begin{center}
\begin{tikzpicture}
\draw[|->] (-4,1) .. controls (-1,0) and (1,0) .. (4,1) node[midway, above]{$\sigma$};
\end{tikzpicture}
\end{center}
\end{figure}
These visual interpretations also give insight into a key difference between $\pi$ and $\sigma$. The map $\pi$ treats the size of its input as the sum of the lengths of the columns, whereas $\sigma$ treats the size of its output as the sum of the lengths of the rows.  This difference will be highlighted again in the generalized bijections defined in Section \ref{section_generalizations}.

Using Propositions \ref{pi_c_representation_proposition} and \ref{c_sigma_prop}, we can more easily describe $\pi\circ\sigma$, and we use this description to prove Theorem \ref{pi_composed_sigma_theorem}.

\begin{proof}[Proof of Theorem \ref{pi_composed_sigma_theorem}]
Let $\phi=(\phi_1,\phi_2,\dots,\phi_r)=[c_1,\dots,c_r]$ be a sequentially congruent partition. Then we have that
\begin{align*}
    \pi(\sigma(\phi))&=\pi(\sigma([c_1,c_2,\dots,c_r]))\\
    &=\pi(\left<1^{c_1}, 2^{c_2},\dots,r^{c_r}\right>)\\
    &=\pi((\underbrace{r,\dots,r}_{c_r \text{ times}}, \underbrace{r-1,\dots,r-1}_{c_{r-1} \text{ times}}, \dots, \underbrace{1,\dots,1}_{c_1 \text{ times}}))\\
    &=[\underbrace{0,\dots,0,1}_{c_r \text{ terms}},\underbrace{0,\dots,0,1}_{c_{r-1} \text{ terms}},\dots,\underbrace{0,\dots,0,1}_{c_1 \text{ terms}}].
\end{align*}
We think of this as moving $c_i$ spaces and then adding a one to the $c$-notation representation of the partition, so if any of the $c_i$ are zero, we simply increment the number in the previous location of an added one. Recall that $c_r$ is nonzero by definition.
We now rewrite $(\pi\circ\sigma)(\phi)$ in frequency notation to obtain
\begin{align*}
\pi(\sigma(\phi))&=\left< \left(\sum_{i = 1}^{r} c_i\right)^{c_1}, \left(\sum_{i = 1}^{r} c_i + \sum_{i = 2}^{r} c_i\right)^{c_2}, \dots,\left(\sum_{i = 1}^{r} c_i + \sum_{i = 2}^{r} c_i + \dots + \sum_{i = r}^{r} c_i \right)^{c_r}\right>\\
&=\left< \left( \sum_{j=1}^{1} \sum_{i = j}^{r} c_i \right)^{c_1}, \left( \sum_{j=1}^{2} \sum_{i = j}^{r} c_i \right)^{c_2}, \dots, \left( \sum_{j=1}^{r} \sum_{i = j}^{r} c_i \right)^{c_r} \right>.
\end{align*}
Substituting $c_i=(\phi_i-\phi_{i+1})/i$ for $1\leq i<r$ and $c_r=\phi_r/r$ yields the desired result.
\end{proof}

It will often be more convenient to think of $\pi\circ\sigma$ in terms of the Young diagram, where we simply perform the ``squish-flip" of $\sigma$ followed by the ``stretch" of $\pi$.
All together, $\pi \circ \sigma$ transforms the Young diagram of a sequentially congruent partition by ``squish-flip-stretch".
These graphical transformations are defined well by our $c$-notation definitions of $\pi$ and $\sigma$ in Propositions \ref{pi_c_representation_proposition} and \ref{c_sigma_prop}, and in Section \ref{section_generalizations} we will apply these same definitions to partitions that are sequentially congruent in a more general sense, with $c$-notation appropriately redefined, to transform the Young diagrams in a similar way.

\section{Generalizations}\label{section_generalizations}

In this section, we define a more general notion of sequentially congruent partitions, and we describe several generalizations of the bijections $\pi$ and $\sigma$.  We also describe several generalizations of a bijection of Schneider, Sellers, and Wagner \cite{schneider2} between sequentially congruent partitions and partitions into $k$th powers.

\subsection{Generalizations of $\s$, $\pi$, and $\sigma$}
In \cite{schneider1}, Schneider and Schneider propose a more general set of frequency congruent partitions based on a set $B=\{b_1,b_2,b_3,\dots\} \subseteq \N$ and a sequence of positive integers $A=(a_1,a_2,a_3,\dots)$.  They define the set of partitions with parts from $B$ where the $b_i$th part occurs a multiple of $a_i$ times, denoted $\f_B(A)$.  Partitions in $\f_B(A)$ are of the form $\langle b_1^{n_1a_1},b_2^{n_2a_2},b_3^{n_3a_3},\dots\rangle$, where all $n_i\geq 0$, and finitely many $n_i$ are nonzero. They then define $\s_B(A)$ as the set of conjugates of partitions in $\f_B(A)$ and state that $\s_B(A)$ is a generalization of $\s$, but they do not further describe the partitions in $\s_B(A)$.  We provide a description here, and then we define generalizations of $\pi$ and $\sigma$ for these generalized sequentially congruent partitions.

The Young diagram of a partition in $\s_B(A)$ consists of rectangles of width $a_i$ and height $b_i$, each occurring $n_i$ times.  These rectangles are analogous to the squares within the Young diagrams of partitions in $\s$.  More explicitly, the Young diagram of a partition in $\s_B(A)$ is of the following form.

\begin{figure}[H]
    \label{general_young_diagram}
    \begin{center}
        \begin{tikzpicture} [scale = 0.3]
        \draw [draw=black] (14,0) rectangle (18,-3);
        
        \draw [draw=black] (22,0) rectangle (26,-3);
        
        \draw [
        thick,
        decoration={
            brace,
            mirror,
            raise = 0.1cm,
        },
        decorate
        ] (22,0) -- (22,-3)
        node [midway, anchor=east, xshift = -1mm] {$b_1$};
        \draw [
        thick,
        decoration={
            brace,
            mirror,
            raise = 0.1cm,
        },
        decorate
        ] (22,-3) -- (26,-3)
        node [midway, anchor=north, yshift = -1mm] {$a_1$};
        \node (dots) at (19, -1.5) {$\cdots$};
        \draw [
        thick,
        decoration={
            brace,
            mirror,
            raise = 0.1cm,
        },
        decorate
        ] (14,-4.3) -- (26,-4.3)
        node [midway, anchor=north, yshift = -1mm] {$n_1$ times};
        
        \draw [draw=black] (0,0) rectangle (5,-4);
        
        \draw [draw=black] (9,0) rectangle (14,-4);
        
        \draw [
        thick,
        decoration={
            brace,
            mirror,
            raise = 0.1cm,
        },
        decorate
        ] (9,0) -- (9,-4)
        node [midway, anchor=east, xshift = -1mm] {$b_2$};
        \draw [
        thick,
        decoration={
            brace,
            mirror,
            raise = 0.1cm,
        },
        decorate
        ] (9,-4) -- (14,-4)
        node [midway, anchor=north, yshift = -1mm] {$a_2$};
        \node (dots) at (6, -2) {$\cdots$};
        \draw [
        thick,
        decoration={
            brace,
            mirror,
            raise = 0.1cm,
        },
        decorate
        ] (0,-5.3) -- (14,-5.3)
        node [midway, anchor=north, yshift = -1mm] {$n_2$ times};
        
        \node (dots) at (-2, -0) {$\cdots$};
        
        \draw [draw=black] (-13,0) rectangle (-4,-7);
        
        \draw [draw=black] (-26,0) rectangle (-17,-7);
        
        \draw [
        thick,
        decoration={
            brace,
            mirror,
            raise = 0.1cm,
        },
        decorate
        ] (-13,0) -- (-13,-7)
        node [midway, anchor=east, xshift = -1mm] {$b_r$};
        \draw [
        thick,
        decoration={
            brace,
            mirror,
            raise = 0.1cm,
        },
        decorate
        ] (-13,-7) -- (-4,-7)
        node [midway, anchor=north, yshift = -1mm] {$a_r$};
        \node (dots) at (-16, -3.5) {$\cdots$};
        \draw [
        thick,
        decoration={
            brace,
            mirror,
            raise = 0.1cm,
        },
        decorate
        ] (-26,-8.3) -- (-4,-8.3)
        node [midway, anchor=north, yshift = -1mm] {$n_r$ times};
        \end{tikzpicture}
    \end{center}
    \caption{Young diagram of a partition $[n_1, \dots, n_r] \in \s_B(A)$}
\end{figure}
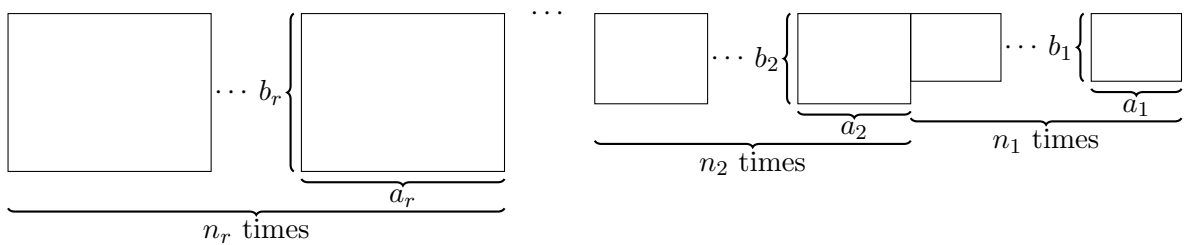

Applying a similar argument to that in the proof of Theorem \ref{c_notation_theorem}, we can show that any $\lambda \in S_B(A)$ can be uniquely written as 
\begin{align}\label{n-notation}
    \lambda &=n_1(\underbrace{a_1,\dots,a_1}_{b_1 \text{ times}})\star n_2(\underbrace{a_2,\dots,a_2}_{b_2 \text{ times}})\star n_3(\underbrace{a_3,\dots,a_3}_{b_3\text{ times}})\star\cdots
\end{align} 
with finitely many nonzero $n_i$.  From this, we can define an analogue of $c$-notation, which we call \textit{$n$-notation} to distinguish between partitions involving squares and partitions involving rectangles.

\begin{definition}
For any partition $\lambda\in\s_B(A)$, we have that $\lambda$ is of the form \eqref{n-notation}, and we write $\lambda = [n_1, n_2,\dots, n_r]_{A,B}$.  We refer to this notation as \emph{$n$-notation}.
\end{definition}

For reference, we can write a general partition $\lambda \in \s_B(A)$ in standard notation as follows:

$$\lambda=[n_1,n_2,\dots,n_r]_{A,B}=\left(\underbrace{\sum_{i=1}^ra_in_i, \dots,\sum_{i=1}^ra_in_i}_{b_1 \text{ times}},\underbrace{\sum_{i=2}^ra_in_i, \dots,\sum_{i=2}^ra_in_i}_{b_2-b_1 \text{ times}},\dots,\underbrace{a_rn_r, \dots,a_rn_r}_{b_r-b_{r-1} \text{ times}}\right).$$ 
Hence $\ell(\lambda)=b_r$ and $|\lambda|=\sum_{i=1}^r a_ib_in_i$. We then observe that $\lambda_{b_i}-\lambda_{b_i+1}=a_in_i$ for each $i$, and thus $n_i=(\lambda_{b_i}-\lambda_{b_i+1})/a_i$. 
Notice also that $\lambda_{b_i} = \lambda_{b_i - 1} = \cdots = \lambda_{b_{i - 1} + 1}$ for all $i > 1$, and we have that
\begin{enumerate}
    \item $\lambda_{b_i}\equiv \lambda_{b_i+1}\equiv \dots \equiv \lambda_{b_{i+1}}\pmod{a_i}$, and
    \item $\lambda_{b_r}\equiv 0 \pmod{a_r}.$
\end{enumerate}

We can see that $\s_B(A)$ is a generalization of $\s$, since $\s = \s_{\N}((1,2,3,\dots))$.
Throughout this section, if $B$ is omitted, then we consider $B = \N$, and if $A$ is omitted, then we consider $A = (1, 2, 3, \dots)$.  We also sometimes refer to $A$ as a subset of $\N$ instead of a sequence of natural numbers when appropriate.

Schneider and Schneider conjectured \cite{schneider1} that there are bijective maps between $\p_A$, the set of partitions of $n$ with parts from $A$ (in this context, $A\subseteq\N$), and $\s_B(A)$ that are analogous to the bijections $\pi$ and $\sigma$ between $\p$ and $\s$.  Recall that $\pi: \p(n) \to \s_{\LG = n}$ and $\sigma: \s _{\LG = n}\to \p(n)$. Furthermore, recall that $\sigma \circ \pi$ is equivalent to conjugation, and $\pi \circ \sigma$ gives an analogue of conjugation specific to sequentially congruent partitions, as defined in Section \ref{section_bijections}.  We define generalizations of $\pi$ and $\sigma$ which relate size and largest part, and we also define different generalizations whose compositions give conjugation and the ``squish-flip-stretch" analogue of conjugation for $\s_B(A)$.  It turns out that there is no generalization that simultaneously has both properties, except in very special cases.

Let $\s_B(A)_{\LG=n}$ be the set of partitions in $\s_B(A)$ with largest part $n$, and let $\p_A(n)$ be the set of partitions in $\p_A$ of size $n$.  If we wish to preserve the exact relationship between size and largest part as in the original maps $\pi$ and $\sigma$, we can define $\sigma_{AB}: \s_B(A)_{\LG = n} \to \p_A (n)$ by $$\sigma_{AB} ([n_1, \dots, n_r]_{A,B}) := \left< a_1^{n_1}, \dots, a_r^{n_r} \right>.$$
In order for this map to be a bijection, however, the terms of $A$ must be distinct, or else $\sigma_{AB}$ is not injective. If $A$ is not increasing, then the terms in the frequency notation will not be in increasing order, but by reordering the terms, the bijection holds. For a sequentially congruent partition $\phi=[n_1,\dots,n_r]_{A,B}$, we see that $|\sigma_{AB}(\phi)|=\phi_1$, since $\phi_1=\sum_{i = 1}^{r} a_i n_i$.

Unfortunately, the domain of the corresponding generalization of $\pi$ is different from the range of $\sigma_{AB}$. Notice that $\sigma_{AB}$ maps into partitions whose parts are from $A$, whereas the generalization of $\pi$ defined below maps from partitions whose columns in the Young diagram are elements of $A$, i.e., the set of conjugates of partitions in $\p_A$. This difference mirrors the key difference mentioned in Section \ref{section_bijections} between $\sigma$ and $\pi$.  The set of conjugates of partitions in $\p_A$ can be written $\s_{A}((1,1,1,\dots))$, and we let $\s_B(A,n)$ denote the set of partitions in $\s_B(A)$ of size $n$, for any $B\subseteq\N$ and any sequence $A$ of natural numbers. We define $\pi_{AB}: \s_{A}((1,1,1,\dots), n) \to \s_B(A)_{\LG = n}$ to transform the $a_i$ column in the Young diagram into a rectangle with width $a_i$ and height $b_i$. Explicitly, we define the map $\pi_{AB}$ by
$$\pi_{AB}((\lambda_1, \dots, \lambda_{a_r})) := [n_1, \dots, n_r]_{A,B},$$
where $n_i = \lambda_{a_i} - \lambda_{a_i + 1}$ for $1 \leq i < r$ and $n_r = \lambda_{a_r}$. This is also a bijection if and only if the terms of $A$ are distinct, and again, we might have to reorder if $A$ is not increasing. 

As mentioned in Section \ref{section_bijections}, we can easily replicate the graphical transformations of $\pi$ and $\sigma$ by simply replacing the $c$s in Propositions \ref{pi_c_representation_proposition} and \ref{c_sigma_prop} with $n$s as defined above. Unfortunately, in doing so, the relationship between size and largest part is no longer preserved in general. However, conjugation and the ``squish-flip-stretch" analogue of conjugation still hold under composition.
Thus we define $\pi'_{AB}: \p \rightarrow{} \s_B(A)$ by \[\pi'_{AB}((\lambda_1,\dots,\lambda_r)):=[n_1, \dots, n_r]_{A,B},\] where $n_i=\lambda_i-\lambda_{i+1}$ for $1 \leq i < r$ and $n_r = \lambda_r$. In standard notation, we have
\begin{equation}\label{pi_general}
\pi'_{AB}((\lambda_1,\dots,\lambda_r))=(\lambda_1',\lambda_2',\dots,\lambda_{b_r}'),
\end{equation}
where $\lambda_i'=a_i\lambda_i+\sum_{j=i+1}^r(a_j-a_{j-1})\lambda_j$ for $b_{i-1}<i\leq b_i$, for all $1\leq i\leq r$, noting that we consider $b_0=0$ for any $B\subseteq \N$.
We define $\sigma'_{AB}: \s_B(A)\rightarrow{} \p$ by \[ \sigma'_{AB} ([n_1, \dots, n_r]_{A,B}):=\left< 1^{n_1}, 2^{n_2},\dots,r^{n_r} \right>,\] so that for any partition $\phi=(\phi_1,\phi_2,\dots,\phi_r)\in\s_B(A)$, we have
\begin{equation}\label{sigma_general}
\sigma'_{AB}((\phi_1,\phi_2,\dots,\phi_r))=\big<1^{(\phi_{b_1}-\phi_{b_2})/a_1}, 2^{(\phi_{b_2}-\phi_{b_3})/a_2},\dots,r^{\phi_{b_r}/a_r}\big>.
\end{equation}
By \eqref{pi_general} and \eqref{sigma_general}, it is clear that 
$$(\sigma'_{AB}\circ\pi'_{AB})(\lambda)=\big\langle 1^{\lambda_1-\lambda_2},2^{\lambda_2-\lambda_3},\dots ,(r-1)^{\lambda_{r-1}-\lambda_r},r^{\lambda_r}\big\rangle,$$
so this composition is still equivalent to conjugation by Proposition \ref{conjugate_closed_form_proposition}.
Additionally, $\sigma'_{AB}$ and $\pi'_{AB}$ act in a similar manner (``squish-flip" and ``stretch", respectively) on the Young diagram of a partition. Thus the composition $\pi'_{AB}\circ \sigma'_{AB}$ is still equivalent to ``squish-flip-stretch," but with $a_i\times b_i$ rectangles instead of $i\times i$ squares.

Although the relation between size and largest part is not usually preserved by $\pi'_{AB}$ and $\sigma'_{AB}$, there is a special case in which size is mapped to a multiple of the largest part and vice versa.

\begin{proposition}
\label{general_size_to_largest_part_special_case_proposition}
Suppose the sequence $A$ is of the form $A = (a, 2a, 3a, \dots )$.  Then we have that $\pi'_{AB}: \p(n) \rightarrow{} \s_B(A)_{\LG=an}$ and $\sigma'_{AB}:\s_B(A)_{\LG=an}\rightarrow{}\p(n)$.

\end{proposition}

\begin{proof}
Note that the largest part of $[n_1,\dots,n_r]_{A,B}\in\s_B(A)$ is $\sum_{i = 1}^{r} a_i n_i$ and the size of the output of $\sigma'_{AB}$ is $n_1 + 2n_2 + \dots + rn_r= \sum_{i = 1}^{r} i n_i$. 
Let $\lambda$ be any partition of $n$, and consider $\pi'_{AB}(\lambda)=[n_1, \dots, n_r]_{A,B}$. We have that
\begin{align*}
|\lambda| &= \lambda_1 + \lambda_2 + \dots + \lambda_r \\
&= \lambda_1 + (-\lambda_2 + \lambda_2) + \lambda_2 + 2(-\lambda_3 + \lambda_3) + \dots + (r-1)(-\lambda_r + \lambda_r) + \lambda_r \\
&= (\lambda_1 - \lambda_2) + 2(\lambda_2 - \lambda_3) + 3(\lambda_3 - \lambda_4) + \dots + r \lambda_r \\
&= n_1 + 2n_2 + \dots + r n_r \\
&= \sum_{i = 1}^{r} i n_i.
\end{align*}
Thus $\pi'_{AB}$ maps a partition of size $\sum_{i = 1}^{r} i n_i$ to a partition with largest part $\sum_{i = 1}^{r} a_i n_i$, and similarly $\sigma'_{AB}$ maps a partition with largest part $\sum_{i = 1}^{r} a_i n_i$ to a partition of size $\sum_{i = 1}^{r} i n_i$.

Now, suppose $f: \N \to \N$ is a map so that $$f \left( \sum_{i = 1}^{\infty} i n_i \right) = \sum_{i = 1}^{\infty} a_i n_i$$ for all partitions $[n_1,n_2,\dots]_{A,B}$. Then $f$ describes a relationship between largest part and size. Since this equality holds for all partitions, consider the partition $[k]$, where $n_1=k\in\mathbb{N}$ and $n_i = 0$ for all $i \geq 2$. Then we have that $f(n_1+2n_2+\cdots)=a_1n_1+a_2n_2+\cdots$, so that $f(k)=a_1k$.  Similarly, if we consider the partition where $n_j = 1$ for some $j \in \N$, and all other $n_i = 0$, we see that $f(j) = a_j$, meaning that $a_j = a_1j$. Therefore, if the largest part of a partition $\phi\in\s_B(A)$ is related to the size of $\sigma'_{AB}(\phi)$, or the size of a partition $\lambda\in\p$ is related to the largest part of $\pi'_{AB}(\lambda)$, then this relationship must be scalar by a factor of $a_1$, and the sequence $A$ must satisfy $a_j = a_1j$ for all $j \in \N$.

Note that if $A$ satisfies $a_j = a_1j$ for all $j \in \N$, then
$$\sum_{i = 1}^{\infty} a_i n_i = \sum_{i = 1}^{\infty} a_1i n_i = a_1 \sum_{i = 1}^{\infty} i n_i.$$
Thus when $A$ is defined as above, the relationship holds.
\end{proof}

\subsection{Partitions into $k$th powers}
In \cite{schneider2}, Schneider, Sellers, and Wagner defined $S(j,k)$ to be the set of partitions $\lambda=(\lambda_1,\lambda_2,\dots,\lambda_r)$ that satisfy \begin{enumerate}
    \item $\lambda_i-\lambda_{i+1}=ji^k$ for $1\leq i\leq r-1$, and
    \item $\lambda_r=ji^r$.
\end{enumerate}
They found that partitions in $S(j,k)$ of size $n$ are in bijection with partitions of $n$ into $(k+1)$th powers where each part occurs exactly $j$ times. As a generalization, we define $S(k)$ to be the subset of partitions $\lambda = (\lambda_1, \lambda_2, \dots, \lambda_r)$ that satisfy \begin{enumerate}
    \item $\lambda_i \equiv \lambda_{i + 1}\pmod{i^k}$, and
    \item $\lambda_r\equiv 0\pmod{r^k}$.
\end{enumerate}
Note that the set of sequentially congruent partitions is the specialization $S(1)$, and in fact all $S(k)$ are actually just $S_\mathbb{N}(\N^k)$ where $\N^k=\{1^k,2^k,3^k,\dots\}$.

We define two bijections, $\sigma_k: \s_B(\N^k)_{\LG = n} \to \p_{\N^k}(n)$ and $\psi_k: \s(\N^k, n) \to \p_{\N^{k + 1}}(n)$, by $$\sigma_k ([n_1, \dots, n_r]) := \big< {(1^k)}^{n_1}, {( 2^k)}^{n_2}, \dots, {( r^k)}^{n_r} \big>$$ and $$\psi_k ([n_1, \dots, n_r]) := \big< {(1^{k+1})}^{n_1}, {( 2^{k + 1} )}^{n_2}, \dots, {( r^{k + 1} )}^{n_r} \big>.$$
Note that $\sigma_1$ is identical to the map $\sigma$ defined by Schneider--Schneider in \cite{schneider1}, and $\psi_1$ is identical to the bijection between $\s(n)$ and $\p_{\N^2}(n)$ given by Schneider--Sellers--Wagner in \cite{schneider2}.

Notice that we can compose $\sigma_{k + 1}$ and $\psi_{k}^{-1}$, meaning that $\s_B(\N^{k + 1})_{\LG = n}$ is in bijection with $\s(\N^k,n)$. This becomes clear when we consider the Young diagrams of the partitions in each set. The rectangles that comprise the Young diagrams of the partitions in $\s_B(\N^{k + 1})_{\LG = n}$ have height $i$ and width $i^{k + 1}$. Since the largest part comes from the sum of all the rectangles' widths, we simply transform the top row of each rectangle into a rectangle of height $i$ and width $i^k$ to get the corresponding partition in $\s(\N^k, n)$. This gives the bijection $\sigma_{k + 1} \circ \psi_{k}^{-1}: \s_B(\N^{k + 1})_{\LG = n} \to \s(\N^k,n)$ defined by $$\sigma_{k + 1} \circ \psi_{k}^{-1} \left( [n_1,\dots,n_r]_{\N^{k+1}, B} \right) = [n_1,\dots,n_r]_{\N^k, \N}.$$

As an extension of this idea, we can define $\eta_{k,p}: \s_B(\N^k)_{\LG = n} \to \s_{ \N^p } ( \N^{k-p}, n)$ for $1 \leq p \leq k$ by
$$\eta_{k,p} \left( [n_1,\dots,n_r]_{\N^k, B} \right) := [n_1,\dots,n_r]_{\N^{k-p}, \N^{p}}.$$
The Young diagrams of partitions in $\s_B(\N^k)_{\LG = n}$ consist of rectangles of width $i^k$, and so $\eta$ maps the top row of each rectangle to a rectangle with area $i^k$. Therefore, $\eta$ maps partitions with largest part $n$ to partitions of size $n$. 

We can define a similar extension $\tau_{k,p,q}: \s_{\N^{p}}(\N^{k-p},n) \to \s_{\N^{q}}(\N^{k-q},n)$ for integers $1 \leq p, q \leq k$ by
$$\tau_{k,p,q}([n_1, \dots, n_r]_{\N^{k-p}, \N^{p}}) := [n_1, \dots, n_r]_{\N^{k-q}, \N^{q}}.$$
This $\tau$ maps rectangles of area $i^k$ to different rectangles with the same area, thus preserving size.

\section{Partition Ideals}\label{section_ideals}

In this section, we change gears and investigate a different question of Schneider and Schneider regarding partition ideals.  Partition ideals were defined by Andrews in the 1970s \cite{andrews1974, andrews} and have been studied more recently by Andrews, Chern, and Li \cite{ACL,C,CL}. In 2019, Schneider and Schneider \cite{schneider1} speculated that frequency congruent partitions, the conjugates of sequentially congruent partitions, form a type of ``quasi-ideal".  Although we do not expand on the notion of ``quasi-ideal" with regard to frequency congruent partitions, this speculation motivated our investigation into the connection between sequentially congruent partitions and partition ideals.  First, we recall the definition of a partition ideal.  For the remainder of the paper, parts with frequency zero are included in the frequency notation representation of all partitions.

\begin{definition}
A subset $\mathscr{I}$ of $\p$ is a \textit{partition ideal} if for any $\lambda\in\mathscr{I}$, removing any number of parts from $\lambda$ yields a new partition $\gamma\in\mathscr{I}$.
\end{definition}

In order to find a partition ideal composed of sequentially congruent partitions, we define the sequence $$\mathcal{A}:=(1, 2, 6, \dots, \lcm(1,\dots,i),\dots)$$ and consider $\s(\mathcal{A})$, the set of all partitions $\lambda=(\lambda_1,\lambda_2,\dots)$ such that $\lambda_i\equiv\lambda_{i+1}\pmod{\lcm(1,\dots,i)}$ for all $i$.  We will use the following lemma to prove that $\s(\mathcal{A})$ is the largest possible subset of sequentially congruent partitions which forms a partition ideal.

\begin{lemma}
\label{s(a)_description_lemma}
A partition $\lambda=(\lambda_1,\lambda_2,\dots,\lambda_r)$ is in $\s(\mathcal{A})$ if and only if each part $\lambda_i$ is divisible by every positive integer less than or equal to $i$.
\end{lemma}
\begin{proof}

Let $\lambda=(\lambda_1,\lambda_2,\dots,\lambda_r) \in \s(\mathcal{A})$. We know that $\lambda_r \equiv 0 \pmod {\lcm (1, \dots, r)}$ by definition. We proceed by induction, assuming that $\lcm (1, \dots, i) \mid \lambda_i$ for some $1<i\leq r$. Since $\lcm (1, \dots, i-1) \mid \lcm (1 \dots, i)$, we have $\lcm (1, \dots, i-1) \mid \lambda_i$. Furthermore, since $\lambda_{i - 1} \equiv \lambda_i \pmod{\lcm (1, \dots, i - 1)}$ by definition, we also have $\lcm (1, \dots, i - 1) \mid \lambda_{i - 1}$.  Thus $\lcm(1, \dots, i) \mid \lambda_i$ for all $i$.

Now, let $\lambda$ be a partition so that $\lcm(1, \dots, i) \mid \lambda_i$ for all $i$. Then for any $i$, $\lambda_i - \lambda_{i + 1}$ is also divisible by $\lcm(1, \dots, i)$, and we have $\lambda_i\equiv\lambda_{i+1}\pmod{\lcm(1,\dots,i)}$. Hence $\lambda \in \s(\mathcal{A})$.
\end{proof}

\begin{proposition}
The set $\s(\mathcal{A})$ is a maximal partition ideal of \SCP s.
\end{proposition}

\begin{proof}
Let $\lambda \in \s(\mathcal{A})$. By Lemma \ref{s(a)_description_lemma}, we have that $\lambda_i\equiv 0 \pmod j$ for all $j\leq i$. Thus, for all $k$ we have  $\lambda_k \equiv \lambda_{k + 1} \equiv 0 \pmod k$, and hence $\lambda$ is sequentially congruent and $\s(\mathcal{A})\subseteq \s$. 

Now consider removing any one part from $\lambda$. Each remaining part is either going to stay at the same index or its index will decrease by one. Since each part is divisible by \textit{every} index less than or equal to its initial index, each part will still be congruent to $0$ modulo its own index and every smaller index. In other words, the resulting partition is still in $\s(\mathcal{A})$.

Now we need to show that $\s(\mathcal{A})$ contains all possible partition ideals of sequentially congruent partitions. Assume by way of contradiction that there exists some $\lambda \notin \s(\mathcal{A})$ which is in a partition ideal of sequentially congruent partitions. Because $\lambda \notin \s(\mathcal{A})$, there is at least one part $\lambda_i$ such that there exists some index $k\leq i$ which does not divide $\lambda_i$. Because $\lambda$ is in a partition ideal of sequentially congruent partitions, we can remove all of the parts after $\lambda_i$, and the resulting partition will still be in this partition ideal. We can also remove the first $i - k$ parts of the partition, so that $\lambda_i$ ends up at index $k$, and this new partition is still in the partition ideal. However, we know that $k \nmid \lambda_i$, so $\lambda_i \not \equiv 0 \pmod k$, and therefore $\lambda$ is not a sequentially congruent partition.
\end{proof}

In the literature, interest in partition ideals has typically been focused on either order 1 partition ideals or linked partition ideals. These types of ideals permit the application of some extremely useful generating function results due to Andrews (see \cite{andrews} for a summary of these results). We will categorize all possible order 1 partition subideals of $\s(\mathcal{A})$ and show that no subideals of $\s(\mathcal{A})$ are linked. We will conclude this section with a discussion of a new class of partition ideals which have infinite order and their potential interest as mathematical objects. We now recall the definition of the order of a partition ideal.

\begin{definition}
A partition ideal $\mathscr{I}$ has \emph{order} $k$ if $k$ is the least positive integer such that for all $\lambda=\langle1^{f_1}, 2^{f_2}, 3^{f_3},\dots\rangle\notin\mathscr{I}$, there exists $m$ such that $\lambda '=\big\langle 1^{f_1 '}, 2^{f_2 '}, 3^{f_3 '},\dots\big\rangle \notin\mathscr{I}$, where 
$$f_i ' = 
\begin{dcases}
      f_i & \text{for }i=m,m+1,\dots,m+k-1, \\
      0   & \text{otherwise}.
\end{dcases}$$
\end{definition}
We can think of the order of a partition ideal as the maximum number of consecutive integers' frequencies we must consider to see that a given partition is not an element of $\mathscr{I}$.

To determine the order of $\s(\mathcal{A})$, observe that $\lambda=(m,1)\notin\s(\mathcal{A})$ while $(m)\in\s(\mathcal{A})$ and $(1)\in\s(\mathcal{A})$ for all $m$. Letting $m$ become arbitrarily large, we have to consider an arbitrarily large number of frequencies of $\lambda$ to obtain $\lambda '\notin\s(\mathcal{A})$. Therefore $\s(\mathcal{A})$ has no well defined order, and so we say that $\s(\mathcal{A})$ has \emph{infinite order}. We can, however, obtain order 1 subideals of $\s(\mathcal{A})$, to which we can apply Andrews' results to get generating functions for each such subideal.  Note that each order 1 subideal $C\subseteq\s(\mathcal{A})$ is determined by a single partition $\lambda\in C$, and all other partitions in $C$ can be obtained by removing one or more parts from $\lambda$. Any such partition ideal is necessarily finite and therefore of limited interest, since the corresponding generating functions devolve into products of finite geometric series.  While $\s(\mathcal{A})$ has order 1 subideals, albeit finite ones, we will show that there are no subideals of $\s(\mathcal{A})$ that are linked. Nevertheless, the discussion surrounding $\s(\mathcal{A})$ and linked partition ideals is a somewhat more interesting one.

We now recall some definitions and examples required to understand the concept of a linked partition ideal.  As in \cite{andrews1974,andrews,CL}, we define the bijection $\varphi$ by $\varphi((\lambda_1, \lambda_2, \dots, \lambda_r)) =  (\lambda_1+1, \lambda_2+1, \dots, \lambda_r+1)$. Next, for a partition ideal $\mathscr{I}$, define $\mathscr{I}^{(m)}=\{\lambda=\langle1^{f_1}, 2^{f_2}, \dots\rangle\in\mathscr{I}\mid f_1=f_2=\dots=f_m=0\}$.

\begin{definition}
A partition ideal $\mathscr{I}$ has \emph{modulus} $m$ if $m$ is a positive integer such that $\varphi^m\mathscr{I}=\mathscr{I}^{(m)}$.
\end{definition}
Note that $\mathscr{I}$ can have more than one modulus. In fact, if $\mathscr{I}$ has modulus $m$, then any multiple of $m$ is also a modulus of $\mathscr{I}$. In order for a partition ideal to be linked, it must have a modulus. We see immediately that $\s(\mathcal{A})$ has no modulus, since for any finite $m$, there is a partition $\lambda\in\s(\mathcal{A})$ of length $m+1$ such that adding $m$ to the last part makes that part no longer a multiple of $m+1$. For example, let $m=4$ and consider $\lambda=(60,60,60,60,60)\in\s(\mathcal{A})$. Note that $(64,64,64,64,64)\notin\s(\mathcal{A})$ since $5 \nmid 64$. Additionally, any order 1 subideal of $\s(\mathcal{A})$ does not have a modulus, since it has a largest allowable part. If $k$ is the largest allowable part of an order 1 partition ideal $I\subseteq\s(\mathcal{A})$, then for any $m$, note that $(k+m)\notin I$, and thus $m$ is not a modulus for $I$.

We can define a subideal of $\s(\mathcal{A})$ with a modulus by limiting the maximum length of the partitions in that subideal. We define $C=\{\lambda\in\s(\mathcal{A})\mid \ell(\lambda)\leq r\}$ so that $C$ is an infinite set. The order of $C$ is again infinite by the same reasoning we used to show that $\s(\mathcal{A})$ has infinite order. To see that $C$ has a modulus, note that $\varphi^mC$ maps surjectively into $C^{(m)}$ when $m$ is a multiple of $\lcm(1,\dots,r)$, meaning $C$ has modulus $m=\lcm(1,\dots,r)$.  In fact, a subideal $C\subseteq\s(\mathcal{A})$ has a modulus if and only if the maximum length of the partitions in $C$ is bounded.

We now define a critically important set in any partition ideal with modulus $m$, and then we state Lemma 8.9 from Andrews \cite{andrews}, which is required to define linked partition ideals. Following the lemma, we give one more useful definition from Chern and Li \cite{CL}, and then we will be sufficiently equipped to define linked partition ideals.

\begin{definition}
For each partition ideal $\mathscr{I}$ with modulus $m$, define
$$L_\mathscr{I}:=\big\{\lambda=\big\langle1^{f_1},2^{f_2},\dots\big\rangle\in\mathscr{I}\mid f_i=0 \text{ for }i>m\big\}.$$
\end{definition}
For two partitions $\lambda, \gamma\in\mathscr{I}$, define $\lambda\oplus\gamma$ to be the partition of length $\ell(\lambda)+\ell(\gamma)$ formed by combining all parts of the two partitions in weakly decreasing order.  Note that $\oplus$ is a different additive operation than $\star$, which we defined in Section \ref{section_intro}.

\begin{lemma}[Andrews, Lemma 8.9]\label{the_Andrews_lemma}
Let $\mathscr{I}$ be a partition ideal with modulus $m$. For each $\lambda\in\mathscr{I}$, we uniquely have
$$\lambda= \pi_1\oplus\big(\varphi^m\pi_2\big)\oplus\big(\varphi^{2m}\pi_3\big)\oplus\dots,$$
where $\pi_i\in L_\mathscr{I}$.
\end{lemma}

\begin{definition}
For any partition $\lambda\in\p$, its \emph{$m$-tail} $\tail_m(\lambda)$ is the collection of parts of $\lambda$ which are at most $m$.

For example,
$$\tail_2((3,3,2,1,1,1))=(2,1,1,1).$$
\end{definition}

Finally, we define a linked partition ideal.

\begin{definition}
\label{linked_definition}
A partition ideal $\mathscr{I}$ is a \emph{linked partition ideal} if 
\begin{enumerate}
    \item $\mathscr{I}$ has a modulus, $m$;
    \item the set $L_\mathscr{I}$ corresponding to $m$ is finite;
    \item for each $\pi\in L_\mathscr{I}$ there corresponds a minimal subset $\mathscr{L}_\mathscr{I}(\pi) \subseteq L_\mathscr{I}$ and a positive integer $l(\pi)$ such that for any $\lambda\in\p$, we have that $\lambda\in\mathscr{I}$ with $\tail_m(\lambda)=\pi$ if and only if we can find a partition $\tilde{\pi}$ with $\tail_m(\tilde{\pi})\in\mathscr{L}_\mathscr{I}(\pi)$ such that
    $$\lambda=\pi\oplus\big(\varphi^{l(\pi)m}\tilde{\pi}\big).$$
    The subset $\mathscr{L}_\mathscr{I}(\pi)$ is called the \emph{linking set of $\pi$}, and the integer $l(\pi)$ is called the \emph{span of $\pi$}.
\end{enumerate}
\end{definition}
Two important observations are that the empty partition (which we will refer to as $\pi_0$) is in the linking set of every element of $L_\mathscr{I}$, and that the linking set of $\pi_0$ is $\mathscr{L}_\mathscr{I}(\pi_0)=L_\mathscr{I}$.

The following two examples can also be found in \cite{CL}.

\begin{example}
\label{linkedEx1}
Let $\mathscr{D}$ denote the set of partitions into distinct parts. Note that $\mathscr{D}$ is a partition ideal with modulus 1, $L_\mathscr{D}=\{\pi_0, (1)\}$, and
\begin{align*}
    \mathscr{L}_\mathscr{D}(\pi_0)&=L_\mathscr{D}, \;\;\;\;&l(\pi_0)=1,\\
    \mathscr{L}_\mathscr{D}((1))&=L_\mathscr{D}, \;\;\;\;&l((1))=1.
\end{align*}
\end{example}

\begin{example}
\label{linkedEx2}
Consider 
$\mathscr{R}$, the set of Rogers--Ramanujan partitions, where adjacent parts differ by at least two. Note that $\mathscr{R}$ is a partition ideal with modulus 1 and that $L_\mathscr{R}=\{\pi_0, (1)\}.$ Then we have
\begin{align*}
    \mathscr{L}_\mathscr{R}(\pi_0)&=L_\mathscr{R}, \;\;\;\;&l(\pi_0)=1,\\
    \mathscr{L}_\mathscr{R}((1))&=L_\mathscr{R}, \;\;\;\;  &l((1))=2.
\end{align*}
Alternatively, we may consider $\mathscr{R}$ to be a partition ideal with modulus 2 and $L_\mathscr{R}=\{\pi_0, (1), (2)\}.$ In this case we have
\begin{align*}
    \mathscr{L}_\mathscr{R}(\pi_0)&=L_\mathscr{R}, \;\;\;\;&l(\pi_0)=1,\\
    \mathscr{L}_\mathscr{R}((1))&=L_\mathscr{R}, \;\;\;\;  &l((1))=1,\\
    \shortintertext{and}\\
    \mathscr{L}_\mathscr{R}((2))&=\{\pi_0, (2)\},\;\;\;\;  &l((2))=1.
\end{align*}
\end{example}

We have shown that we can construct ideals $C\subseteq\s(\mathcal{A})$ with modulus $m$ where $m=\lcm (1, \dots, r)$ by choosing all elements of $\s(\mathcal{A})$ whose length is no longer than $r$ so that condition (1) from Definition \ref{linked_definition} is satisfied. Observe that such a subset $C$ also satisfies condition (2), since the length of the partitions in $C$ must be bounded in order for $C$ to have a modulus. However, we will show that condition (3) does not hold for any such subset $C$.

\begin{proposition}\label{subideals_not_linked}
Any subideal $C\subseteq\s(\mathcal{A})$ is not a linked partition ideal.
\end{proposition}

\begin{proof}
Let $C\subseteq\s(\mathcal{A})$ be a subideal, and suppose the length of the partitions in $C$ is bounded by $r$ for some $r\geq2$, so that C has modulus $m=\lcm(1,\dots,r)$. Since $2\leq\lcm(1,\dots, r)$, we have $(2)\in L_C$. Note that $(\lcm(1,\dots,r)+2, 2)=(m+2,2)\in C$, and thus we have $(2)\in\mathscr{L}_C((2))$ and $l((2))=1$. Therefore $(2m+2, m+2, 2)\in C$, but $2\not\equiv 0 \pmod 3$, which contradicts the fact that $C$ is a subset of $\s(\mathcal{A})\subseteq\s$. Thus no such subideal $C$ is linked.
\end{proof}

Since $\s(\mathcal{A})$ and its subideals are not linked, we cannot apply the generating function results of Andrews to sequentially congruent partitions.  However, there has been no other research so far on partition ideals of infinite order, so here we describe several more partition ideals of infinite order and ask whether they are linked.

\begin{example}
\label{infinite_order_ex1}
Let $\mathscr{R}'$ denote the set of all partitions with no parts below their Durfee square. A similar argument to the one for $\s(\mathcal{A})$ above verifies that $\mathscr{R}'$ is an infinite order partition ideal. Additionally, $\mathscr{R}'$ is not linked by a similar argument to that in the proof of Proposition \ref{subideals_not_linked}. By a simple Young diagram transformation, we know that $\mathscr{R}'$ is in bijection with the set of Rogers--Ramanujan partitions $\mathscr{R}$ from Example \ref{linkedEx2}. It is straightforward to confirm that $\mathscr{R}$ has order 2, which leads to the interesting observation that infinite order partition ideals can be in bijection with finite order partition ideals.
Furthermore, note that $\s(\mathcal{A})\subseteq\s\subseteq\mathscr{R}'$, since for any partition $\phi=(\phi_1, \phi_2, \dots,\phi_r)\in\s$, we have that $\phi_r=kr$ for some $k\geq1$. Thus $\phi_r\geq \ell(\phi)$ and therefore $\phi\in\mathscr{R}'$.
\end{example}

\begin{example}
\label{infinite_order_ex2}
Let $\mathscr{A}$ be the set of all partitions $(\lambda_1,\dots,\lambda_r)$ that satisfy $\lambda_{r-i} - \lambda_{r-i+1} \geq i$ for all $1\leq i\leq r-1$. Consider $\lambda=(k+1,k,1)$. Note that $\lambda_1-\lambda_2=1\not\geq 2$ and thus $\lambda\notin\mathscr{A}$. Letting $k$ be arbitrarily large, we see that $\mathscr{A}$ has no finite order. Notice that $\mathscr{A}$ has modulus 1, and $L_\mathscr{A}=\{\pi_0, (1)\}$. However, the subpartition $(1)$ does not have a consistent span, since $(2,1)\in\mathscr{A}$ but $(3,2,1)\notin\mathscr{A}$. Therefore $\mathscr{A}$ is not linked.
\end{example}

\begin{example}
\label{infinite_order_ex3}
Let $\mathscr{N}$ be the set of all partitions with length at most $n$. This is an infinite order partition ideal with modulus 1, and the set $L_\mathscr{N}$ is finite with cardinality $n+1$. However, $\mathscr{N}$ is not linked, since we can construct partitions with length greater than $n$ using the representation described in condition (3) of Definition \ref{linked_definition}.
\end{example}

\begin{example}
\label{infinite_order_ex4}
Let $\mathscr{P}$ denote the set of all partitions whose parts have the same parity. Notice that $\mathscr{P}$ is also an infinite order partition ideal with modulus 1. This ideal is not linked because the set $L_\mathscr{P}$ is not finite.  These parity partitions can be generalized to sets of partitions whose parts are all the same modulo $k$ for some $k\geq 2$. Each such set is a partition ideal with infinite order which is not linked.
\end{example}

After finding that none of these infinite order partition ideals are linked, we prove the following.

\begin{theorem}
\label{linked_theorem}
Any linked partition ideal has finite order.
\end{theorem}

\begin{proof}
Suppose $C$ is a linked partition ideal and we have a partition $\lambda\notin C$. Then there is no representation of $\lambda$ as in Lemma \ref{the_Andrews_lemma} that satisfies the requirements of Definition \ref{linked_definition}. If $\tail_m(\lambda)\notin L_C$, then we choose
$$\lambda '=\tail_m(\lambda)\oplus \varphi^m\pi_0\oplus\varphi^{2m}\pi_0\oplus\cdots\notin C$$
so that the frequencies of the first $m$ natural numbers in $\lambda'$ match those of $\lambda$ and the rest are 0. If $\tail_m(\lambda)\in L_C$, then we have a representation of the form 
$$\lambda= \pi_1\oplus \varphi^m \pi_2 \oplus \cdots \oplus \varphi^{m(b-1)}\pi_b,$$
where each $\varphi^{m(i-1)}\pi_i$ satisfies the required linking set properties up until $\varphi^{m(b-1)}\pi_b$. Notice that $\varphi^{m(b-1)}\tail_m\pi_b$ corresponds to $m$ frequencies of $\lambda$. If $\tail_m\pi_b\notin L_C$, then we choose 
$$\lambda ' =\pi_0\oplus \varphi^m\pi_0\oplus\cdots\oplus\varphi^{m(b-2)}\pi_0\oplus\varphi^{m(b-1)}\tail_m\pi_b\oplus\varphi^{mb}\pi_0\oplus\cdots \notin C,$$
where again $m$ consecutive frequencies of $\lambda'$ match those of $\lambda$ and the rest are 0. Otherwise, it must be that $\pi_b$ is the first $\pi_i$ where $\tail_m(\pi_b)\neq \pi_0$ after some nonempty partition $\pi_a$. If $\tail_m\pi_b\in \Ll_C(\pi_a)$, then 
$$\lambda= \pi_1\oplus \varphi^m \pi_2 \oplus \cdots \oplus\varphi^{m(a-1)}\pi_a\oplus\underbrace{\varphi^{ma}\pi_0\oplus\cdots\oplus\varphi^{m(b-2)}\pi_0}_{<l(\pi_a)-1\text{ times}} \oplus \varphi^{m(b-1)}\pi_b,$$
and we can choose 
\begin{align*}
    \lambda ' &=\pi_0\oplus\cdots\oplus\varphi^{m(a-2)}\pi_0\oplus\varphi^{m(a-1)}\pi_a\oplus\\
    &\hspace{1.5cm}\underbrace{\varphi^{ma}\pi_0\oplus\cdots\oplus\varphi^{m(b-2)}\pi_0}_{<l(\pi_a)-1\text{ times}}\oplus\varphi^{m(b-1)}\tail_m\pi_b\oplus\varphi^{mb}\pi_0\oplus\cdots \notin C,
\end{align*}
which corresponds to less than $(l(\pi_a)+1)m$ consecutive frequencies matching those of $\lambda$ and the rest equal to 0. Otherwise, we have that $\tail_m\pi_b\notin \Ll_C(\pi_a)$, and we can choose 
\begin{align*}
    \lambda ' &=\pi_0\oplus\cdots\oplus\varphi^{m(a-2)}\pi_0\oplus\varphi^{m(a-1)}\pi_a\oplus\\
    &\hspace{1.5cm}\underbrace{\varphi^{ma}\pi_0\oplus\cdots\oplus\varphi^{m(b-2)}\pi_0}_{\leq l(\pi_a)-1\text{ times}}\oplus\varphi^{m(b-1)}\tail_m\pi_b\oplus\varphi^{mb}\pi_0\oplus\cdots \notin C,
\end{align*}
which corresponds to at most $(l(\pi_a)+1)m$ consecutive frequencies matching those of $\lambda$ and the rest 0. Since $l(\pi_a)\leq \max(\{l(\pi_i):\pi_i\in L_C\})$, we have that $C$ has order at most $(\max\{l(\pi_i):\pi_i\in L_C\}+1)m$, and therefore $C$ has finite order.
\end{proof}

Notice that the partition ideals in Example \ref{infinite_order_ex4}, and the suggested generalizations thereof, can be thought of as unions of order 1 partition ideals whose intersections contain only the empty partition. For this reason, finding generating functions for such partition ideals is simple, even without Andrews' generating function results. For example, the generating function for $\mathscr{P}$ as defined in Example \ref{infinite_order_ex4} is simply
$$\sum_{n=0}^\infty p(\mathscr{P},n)q^n=\prod_{i=1}^\infty \frac{1}{1-q^{2n-1}}+\prod_{i=1}^\infty \frac{1}{1-q^{2n}}-1,$$
where we subtract 1 to avoid double counting the empty partition. It does not appear that Examples \ref{infinite_order_ex1}, \ref{infinite_order_ex2}, and \ref{infinite_order_ex3} are similarly composed of order 1 partition ideals.

Tweaking the definition of the order of a partition ideal provides further evidence that Examples \ref{infinite_order_ex1}, \ref{infinite_order_ex2}, \ref{infinite_order_ex3}, and \ref{infinite_order_ex4} fall into different classes.  For example, suppose we say that a partition ideal $\mathscr{I}$ has ``weak order" $k$ if $k$ is the least positive integer such that for all $\lambda=\Big\langle\lambda_1^{f_{\lambda_1}}, \lambda_2^{f_{\lambda_2}}, \lambda_3^{f_{\lambda_3}},\dots\Big\rangle\notin\mathscr{I}$, there exists $m$ such that $\lambda '=\Big\langle \lambda_1^{f_{\lambda_1}'}, \lambda_2^{f_{\lambda_2}'}, \lambda_3^{f_{\lambda_3}'},\dots\Big\rangle \notin\mathscr{I}$, where 
$$f_{\lambda_i}' = 
\begin{dcases}
      f_{\lambda_i} & \text{for }i=m,m+1,\dots,m+k-1, \\
      0   & \text{otherwise}.
\end{dcases}$$
In other words, weak order omits parts whose frequencies are zero so that the $k$ consecutive parts need not be consecutive positive integers.  Under this new definition, the ideals $\mathscr{R}'$ and $\mathscr{A}$ from Examples \ref{infinite_order_ex1} and \ref{infinite_order_ex2} still have infinite weak order, while $\mathscr{N}$ and $\mathscr{P}$ from Examples \ref{infinite_order_ex3} and \ref{infinite_order_ex4} have weak order $n+1$ and $2$, respectively.

To provide even further evidence of distinctions among these partition ideals, we can tweak the definition of a linked partition ideal to allow for the span of $\pi_0$ to be 1 before a nonempty partition appears in the representation given in Lemma \ref{the_Andrews_lemma}, and to change after such a partition appears in the representation.
If we call partitions satisfying this definition ``quasi-linked", note that the ideals discussed in Examples \ref{infinite_order_ex1}, \ref{infinite_order_ex2}, and \ref{infinite_order_ex3} are not quasi-linked.
Furthermore, we see that $\mathscr{P}$ from Example \ref{infinite_order_ex4} and its suggested generalizations are also not quasi-linked, as $L_\mathscr{P}$ is not finite.
However, this issue can be remedied by taking $\mathscr{P}'\subseteq\mathscr{P}$ to be the set of all partitions in $\mathscr{P}$ with distinct parts (or parts occurring no more than $q$ times for some $q\in\N$).
In this case, we do have that $\mathscr{P}'$ and its generalizations are quasi-linked, since $l(\pi_0)$ can be defined as 1 before a nonempty partition occurs in the Lemma \ref{the_Andrews_lemma} representation of a partition in $\mathscr{P}'$, and 2 (or $k$ in the generalizations) after a nonempty partition occurs in the representation.

These observations, paired with the difficulty of defining properties intrinsic to infinite order partition ideals, have led us to think of infinite order as the lack of known exploitable properties relating to partition ideals, rather than a property worthy of research in its own right.
However, the various characteristics of the examples given above do suggest the potential for further categorization of partition ideals.
Such categorization may yield new results, such as ways to find generating functions for partition ideals that are unions of order 1 partition ideals, or opportunities to adjust certain key definitions relevant to partition ideals so as to be more broadly applicable.

\section*{Acknowledgments}
The authors were supported by a National Science Foundation grant (DMS-2149921).

\end{document}